\documentclass{amsart}

\usepackage{amsmath}
\usepackage{amssymb}
\usepackage{amsthm}
\usepackage{thmtools}
\usepackage{mathrsfs}
\usepackage{tikz}
\usepackage{enumerate}
\usepackage{multirow}

\usepackage{hyperref}
\usepackage{cleveref}
\usepackage{orcidlink}

\newcommand{\CC}{\mathbb C}
\newcommand{\II}{\mathscr I}

\newcommand{\kk}{\Bbbk}
\newcommand{\LL}{\mathfrak L}
\newcommand{\OO}{\mathscr O}
\newcommand{\PP}{\mathbb P}
\newcommand{\pp}{\mathfrak p}
\newcommand{\RR}{\mathbb R}
\newcommand{\ZZ}{\mathbb Z}

\newcommand{\vb}[1]{%
  \ifcat\noexpand#1\relax\boldsymbol{#1}\else\mathbf{#1}\fi}

\DeclareMathOperator{\aff}{aff}

\DeclareMathOperator{\expdim}{expdim}
\DeclareMathOperator{\He}{He}
\DeclareMathOperator{\lcm}{lcm}
\DeclareMathOperator{\relint}{relint}
\DeclareMathOperator{\Res}{Res}

\DeclareMathOperator{\Ter}{Ter}
\DeclareMathOperator{\Tr}{Tr}

\declaretheorem[numberwithin=section]{theorem}
\declaretheorem[sibling=theorem]{lemma}

\declaretheorem[style=definition,sibling=theorem]{definition}
\declaretheorem[style=definition,sibling=theorem]{example}

\declaretheorem[style=remark,sibling=theorem]{remark}

\title{Terracini loci of tangential varieties of Veroneseans}
\author{Douglas A. Torrance \orcidlink{0000-0003-3999-4973}}
\address{Department of Mathematical Sciences, Piedmont University, Demorest, GA 30535 USA}
\address{School of Mathematics, Georgia Institute of Technology, Atlanta, GA 30332 USA}
\email{dtorrance9@gatech.edu}


\begin{document}

\begin{abstract}
  The Terracini locus of a projective variety tells us the sets of points at which Terracini's lemma will fail to recover the dimension of the corresponding secant variety.  We study the Terracini loci of tangential varieties of Veronese varieties, classifying several cases, including a complete classification of the Terracini loci of tangential varieties to rational normal curves.  We also show that many Terracini loci involve only sets with ``intratangent pairs'' that lie on the same tangent space to the underlying Veronese variety.  Finally, we give some applications to algebraic statistics, showing that osculating varieties of Veronese varieties are moment varieties for local mixtures of location families and connecting Terracini loci to local identifiability.
\end{abstract}

\maketitle

\section{Introduction}

One of the most important tools in studying secant varieties is Terracini's lemma \cite{T11}, which tells us that we can study the $r$th secant variety of a projective variety by studying the span of its embedded tangent spaces at $r$ general smooth points.  But what if our points are not general?  The set of configurations of $r$ smooth points for which this span gives us something smaller than expected has become known as the $r$th \textit{Terracini locus} of the variety, and the study of these loci has become an active area of research in recent years \cite{B22b,B23,B26,BB24,BB24b,BC21,BC24,BBF25,BBS23,BV23,CG24,GSTT25,LM24}.

The natural first places to study Terracini loci are the projective varieties where the story of secant nondefectivity has been settled.  Consequently, much attention has been given to the Terracini loci of Veronese varieties, whose secant varieties have been known to be nondefective (in most cases) for over three decades thanks to Alexander and Hirschowitz \cite{AH95}.

Tangential varieties of Veronese varieties have also been an object of study in the secant nondefectivity literature.  A conjecturally complete list of defective cases was presented by Catalisano, Geramita, and Gimigliano \cite{CGG02}, Ballico verified it for ternary and quaternary forms \cite{B05}, and Bernardi, Catalisano, Gimigliano, and Id\`a established an inductive argument showing that if the conjecture was true for cubics, then it was true in general \cite{BCGI09}.  Abo and Vannieuwenhoven then verified the base cases, finishing the story \cite{AV18}.

In this paper, we begin to describe the Terracini loci of these varieties.  In \Cref{sec: chow-veronese varieties and tangent spaces}, we introduce a family of projective varieties, the \textit{Chow-Veronese varieties}, to which the Veronese varieties and their tangentials belong, and describe their tangent spaces.  In \Cref{sec: secant varieties and terracini loci}, we establish a number of facts about the Terracini loci of tangentials of Veronese varieties using elementary linear algebra.  Then in \Cref{inverse system section}, we review the correspondence established in \cite{CGG02} between the tangent spaces of these varieties and a certain family of zero-dimensional schemes, giving us an alternate description of our Terracini loci.  We take a short interlude in \Cref{sec: binary forms} to completely describe the Terracini loci of tangential varieties of rational normal curves, which we generalize in \Cref{sec: horace} to more variables via induction using the Horace lemma.

We close the paper with some applications of these concepts to algebraic statistics, building on an observation made by Grosdos and Wageringel relating tangential varieties of Veronese varieties to local mixtures \cite{GKW20}.  In \Cref{sec: moment varieties}, we review various statistical concepts, including moment varieties, mixtures, and local mixtures, and show that the moment varieties of local mixtures of location families are osculating varieties of Veroneseans.  Finally, in \Cref{sec: terracini loci of moment varieties}, we relate this to Terracini loci, establishing statistical versions of the Alexander-Hirschowitz and Abo-Vannieuwenhoven theorems and connecting Terracini loci to local identifiability.

\section{Chow-Veronese varieties and their tangent spaces}\label{sec: chow-veronese varieties and tangent spaces}

Let $\kk$ be an algebraically closed field of characteristic 0 and consider the polynomial ring $S=\kk[y_0,\ldots,y_n]$ with the usual grading.  For any positive integer $d$, let $S_d$ be the $d$th graded piece of $S$ as a $\kk$-vector space, and note that projective space $\PP(S_d)$ has dimension $\binom{n+d}{d}-1$.  In particular, $\PP(S_1)=\PP^n$.  We denote by $[g]$ the point in $\PP(S_d)$ corresponding to the $(n+1)$-ary $d$-ic $g\in S_d$.

\begin{definition}
  Given a positive integer $d$ and a partition $\vb d=(d_1,\ldots,d_k)$ of $d$, i.e., $d=d_1+\cdots + d_k$ and $d_1\geq \cdots \geq d_k\geq 1$, consider the map $\nu_{\vb d}:(\PP^n)^{\times k}\to\PP(S_d)$ defined by
  \begin{equation*}
    ([\ell_1],\ldots,[\ell_k])\mapsto[\ell_1^{d_1}\cdots\ell_k^{d_k}].
  \end{equation*}
  The image of this map is known as a \textit{Chow-Veronese variety}, which, abusing notation, we will denote by $\nu_{\vb d}(\PP^n)$.
\end{definition}

Note that by construction, $\dim\nu_{\vb d}(\PP^n)=kn$.  This family of varieties gets its name from the two extreme cases, the \textit{Chow variety} $\nu_{(1,\ldots,1)}(\PP^n)$ and the \textit{Veronese variety}, or \textit{Veronesean}, $\nu_{(d)}(\PP^n)$, which we will also denote $\nu_d(\PP^n)$.

For a projective variety $X$ and a point $p\in X$, we denote by $T_pX$ the embedded projective tangent space to $X$ at $p$ and by $\widehat T_pX$ the affine cone over $T_pX$.  The following result is widely used in the study of Chow-Veronese varieties.

\begin{lemma}[{\cite[Proposition 2.2]{CCGO17}}]
  \label{tangent space - chow-veronese}
  Consider $\ell_1,\ldots,\ell_k\in S_1$ and let $p=[\ell_1^{d_1}\cdots\ell_k^{d_k}]\in\nu_{\vb d}(\PP^n)$.  If $[\ell_1],\ldots,[\ell_k]$ are distinct points in $\PP^n$, then $p$ is smooth and
  \begin{equation*}
    \widehat T_p\nu_{\vb d}(\PP^n) = \ell_1^{d_1-1}\cdots\ell_k^{d_k-1}\sum_{i=1}^k\frac{\ell_1\cdots\ell_k}{\ell_i}S_1.
  \end{equation*}
\end{lemma}

Some work has been done on describing the singular loci of Chow-Veronese varieties, and in particular the $n = 1$ case, which are also known as \textit{coincident root loci} \cite{C03,K12}.

Recall that for a smooth projective variety $X$ such as $\nu_d(\PP^n)$, its \textit{tangential variety} (or \textit{tangential developable}) is
\begin{equation*}
  \tau(X) = \bigcup_{p\in X}T_pX.
\end{equation*}

\begin{remark}\label{osculating variety}
  The tangential variety is the first nontrivial example of an \textit{osculating variety}, which is the (Zariski closure of the) union of the \textit{osculating spaces} of $X$.  In particular, if the affine cone $\widehat X$ of a smooth projective variety $X$ is locally parametrized at $p\in X$ by the smooth map $\varphi$, and if $[\varphi(q)]=p$, then the $s$th osculating space to $X$ at $p$ is
  \begin{equation*}
    T^{(s)}_pX = \PP\left(\left\langle\partial^{\vb\beta}\varphi(q):|\vb\beta|\leq s\right\rangle\right),
  \end{equation*}
  where for each vector $\vb\beta=(\beta_0,\ldots,\beta_{\dim X})$ of nonnegative integers, $\partial^{\vb\beta}$ is the standard partial derivative operator and $|\vb\beta|=\beta_0+\cdots+\beta_{\dim X}$.  Then the $s$th osculating variety of $X$ is
  \begin{equation*}
    \tau_s(X) = \overline{\bigcup_{p\in X}T^{(s)}_pX}.
  \end{equation*}
  We will return to osculating varieties in \Cref{sec: moment varieties}.
\end{remark}

By \Cref{tangent space - chow-veronese}, if $d\geq 2$, then $\tau(\nu_d(\PP^n))$ consists of points of the form $[\ell^{d-1}m]$ for $\ell,m\in S_1$, and so $\tau(\nu_d(\PP^n))=\nu_{(d-1,1)}(\PP^n)$.  In particular, $[\ell^{d-1}m]$ lies on a tangent line to $\nu_d(\PP^n)$ at $[\ell^d]$.

\begin{definition}
  Given a point $[\ell^{d-1}m]\in\tau(\nu_d(\PP^n))$ with $d\geq 3$, we will call the point $[\ell]\in\PP^n$ the \textit{base point} and $[m]\in\PP^n$ the \textit{direction point}.
\end{definition}

The singular points of $\tau(\nu_d(\PP^n))$ are exactly the points of $\nu_d(\PP^n)$ (except for $(d,n)=(2,1)$, where the tangential variety fills the ambient space), and so $[\ell^{d-1}m]$ is a smooth point of $\tau(\nu_d(\PP^n))$ precisely when $\ell$ and $m$ are not scalar multiples of one another.

Applying \Cref{tangent space - chow-veronese} again, we obtain the following special case.

\begin{lemma}\label{tangent space - tangential}
  If $d\geq 2$ and $p=[\ell^{d-1}m]\in\tau(\nu_d(\PP^n))$ with $[\ell]\neq[m]$, then
  \begin{equation*}
    \widehat T_p\tau(\nu_d(\PP^n)) = \ell^{d-2}(\ell S_1 + m S_1).
  \end{equation*}
\end{lemma}

We close this section with an observation that distinct points of $\tau(\nu_d(\PP^n))$ may share a tangent space.  In particular, this occurs when they lie on the same tangent line to $\nu_d(\PP^n)$.

\begin{lemma}\label{shared tangent space}
  Suppose $\ell,m_1,m_2\in S_1$ and $d\geq 2$ such that $[\ell^{d-1}m_1]$ and $[\ell^{d-1}m_2]$ are smooth and distinct.  Then $T_{[\ell^{d-1}m_1]}\tau(\nu_d(\PP^n))=T_{[\ell^{d-1}m_2]}\tau(\nu_d(\PP^n))$ if and only if $[\ell]$, $[m_1]$, and $[m_2]$ are collinear.
\end{lemma}

\begin{proof}
  By \Cref{tangent space - tangential}, we see that $\widehat T_{[\ell^{d-1}m_1]}\tau(\nu_d(\PP^n)) = \ell^{d-2}LS_1$, where $L=\langle\ell, m_1\rangle$.  So we have the same tangent space when we also have $L=\langle\ell, m_2\rangle$, which occurs if and only if $[\ell]$, $[m_1]$, and $[m_2]$ all lie on the line $\PP(L)$.
\end{proof}

\section{Secant varieties and Terracini loci}\label{sec: secant varieties and terracini loci}

\begin{definition}
  Let $X\subset\PP^N$ be a variety.  For a positive integer $r$, the $r$th \textit{secant variety} of $X$ is the Zariski closure of the union of all the secant $(r-1)$-planes to $X$, i.e.,
\begin{equation*}
  \sigma_r(X) = \overline{\bigcup\{\langle p_1,\ldots,p_r\rangle:p_1,\ldots,p_r\in X\}}.
\end{equation*}
\end{definition}

The following result is an extremely important tool for dealing with secant varieties.

\begin{lemma}[Terracini's lemma \cite{T11}]
  If $p_1,\ldots,p_r\in X$ are smooth points in general position, then 
\begin{equation*}
  \dim\sigma_r(X) = \dim\left\langle T_{p_1}X,\ldots,T_{p_r}X\right\rangle.
\end{equation*}
\end{lemma}

\begin{definition}
  The \textit{expected dimension} of $\sigma_r(X)$ is
  \begin{equation*}
    \expdim\sigma_r(X) = \min\{r(\dim X + 1), N + 1\} - 1,
  \end{equation*}
  and its \textit{defect} is
  \begin{equation*}
    \delta_r(X) = \expdim\sigma_r(X) - \dim\sigma_r(X).
  \end{equation*}
  A secant variety with nonzero defect is \textit{defective}.  Otherwise, it is \textit{nondefective}.
\end{definition}

Looking at the statement of Terracini's lemma, we might wonder about our choice of $p_1,\ldots,p_r$.  If the secant variety is nondefective, then are there any configurations of $r$ smooth points that are not ``general enough'' for us to recover the dimension using their tangent spaces?  This leads us to the following definition.

\begin{definition}\label{definition of terracini locus}
  Let $X$ be a projective variety and $r$ a positive integer.  If $X_{sm}^{(r)}$ consists of the sets of $r$ distinct smooth points of $X$, then the $r$th \textit{Terracini locus} of $X$ is the variety
\begin{equation*}
  \Ter_r(X) = \overline{\left\{\{p_1,\ldots,p_r\}\in X_{sm}^{(r)}:\dim\left\langle T_{p_1}X,\ldots,T_{p_r}X\right\rangle<r(\dim X + 1)-1\right\}},
\end{equation*}
where the Zariski closure is taken inside $X_{sm}^{(r)}$.
\end{definition}

Now let us focus on the Terracini loci of the tangential varieties of Veroneseans.  The first place to look is for the cases where $\Ter_r(\tau(\nu_d(\PP^n)))=\tau(\nu_d(\PP^n))_{sm}^{(r)}$, which will happen in exactly two situations:
\begin{enumerate}[(a)]
\item If $\sigma_r(\tau(\nu_d(\PP^n)))$ is defective, then the inequality in the definition will be true for all sets of $r$ smooth points in general position, which is a dense subset of $\tau(\nu_d(\PP^n))_{sm}^{(r)}$.
\item If $r$ is \textit{superabundant}, i.e., $r>\frac{1}{2n+1}\binom{n+d}{d}$, then the inequality in the definition will always be true since the span of the tangent spaces can at most fill $\PP(S_d)$.
\end{enumerate}

The defectivity of secant varieties of tangential varieties of Veroneseans has been completely classified \cite{AV18}.  In particular, $\sigma_r(\tau(\nu_d(\PP^n)))$ is always nondefective unless
\begin{enumerate}[(a)]
  \item $d=2$, $n\geq 4$, and $r\in\left\{2,\ldots,\left\lfloor\frac{n}{2}\right\rfloor\right\}$ \cite[Proposition 3.3(ii)]{CGG02} or
  \item $d=3$, $n\in\{2,3,4\}$, and $r=n$ \cite[Propositions 3.2(i) and 3.4]{CGG02}.
  \end{enumerate}

  Combining the defective and superabundant cases and noting that there is some redundancy (i.e., we may remove from the list of defective cases the ones that are also superabundant), we now state these results using the language of Terracini loci.

\begin{lemma}[Abo-Vannieuwenhoven]\label{tangential secant variety defectivity}
  $\Ter_r(\tau(\nu_d(\PP^n)))=\tau(\nu_d(\PP^n))_{sm}^{(r)}$ if and only if
  \begin{enumerate}[(a)]
  \item $d=2$, $n\geq 6$, and $r\in\left\{2,\ldots,\left\lfloor\frac{1}{2n+1}\binom{n+2}{2}\right\rfloor\right\}$,
  \item $(d,n,r)=(3,2,2)$, or
  \item $r>\frac{1}{2n+1}\binom{n+d}{d}$.
  \end{enumerate}
\end{lemma}

We now turn to the cases where $r$ is \textit{subabundant}, i.e., $r\leq\frac{1}{2n+1}\binom{n+d}{d}$, and $\sigma_r(\tau(\nu_d(\PP^n)))$ is nondefective.  Thanks to \Cref{tangent space - tangential}, we can say the following.

\begin{lemma}\label{tangential terracini locus - vector space}
  Suppose $d\geq 2$ and $\Gamma=\{[\ell_1^{d-1}m_1],\ldots,[\ell_r^{d-1}m_r]\}$ is a set of smooth points in $\tau(\nu_d(\PP^n))$.  Then $\Gamma\in\Ter_r(\tau(\nu_d(\PP^n)))$ if and only if
  \begin{equation*}
    \dim\sum_{i=1}^r\ell_i^{d-2}(\ell_iS_1+m_iS_1)<r(2n+1).
  \end{equation*}
\end{lemma}

This result immediately leads to an inductive technique that will be useful.  It is essentially a special case of \cite[Propositions 3.1 and 3.2]{GSTT25}.  But first, note that if $P\subset\PP^n$ with $P\cong\PP^k$, i.e., there exists a $(k+1)$-dimensional subspace $V\subset S_1$ such that $P=\PP(V)$, then we define $\Ter_r(\tau(\nu_d(P)))$ to be the set of all sets of $r$ smooth points in $\tau(\nu_d(P))$ whose Zariski tangent spaces in $\PP\left(S^dV\right)$ fail to span a linear space with the expected dimension, i.e., all $\{[\ell_1^{d-1}m_1],\ldots,[\ell_r^{d-1}m_r]\}$ for which $\ell_i,m_i\in V$ and
\begin{equation}\label{terracini locus of a subspace}
  \dim\sum_{i=1}^r\ell_i^{d-2}(\ell_iV + m_iV)<r(2k+1).
\end{equation}
We see that $\Ter_r(\tau(\nu_d(P)))\cong\Ter_r(\tau(\nu_d(\PP^k)))$ via any isomorphism between $P$ and $\PP^k$.

\begin{lemma}\label{tangential terracini locus - induction via vector spaces}
  Suppose $P\subset\PP^n$ is a linear space, $\Gamma$ is a set of $r$ smooth points in $\tau(\nu_d(\PP^n))$, and consider $\Gamma'\subset\Gamma\cap\tau(\nu_d(P))$ with $|\Gamma'|=s$.  If $\Gamma'\in\Ter_s(\tau(\nu_d(P)))$, then $\Gamma\in\Ter_r(\tau(\nu_d(\PP^n)))$.
\end{lemma}

\begin{proof}
  Suppose $P=\PP(V)$ for some $(k+1)$-dimensional subspace $V\subset S_1$.  Letting $V'$ be a complement of $V$ (i.e., $V\oplus V'=S_1$), we have
  \begin{align*}
    \sum_{i=1}^r\ell_i^{d-2}(\ell_iS_1 + m_iS_1) &= \sum_{i=1}^s\ell_i^{d-2}(\ell_iV + m_iV) + \sum_{i=1}^s\ell_i^{d-2}(\ell_iV' + m_iV') \\
                                                       &\quad + \sum_{i=s+1}^r\ell_i^{d-2}(\ell_iS_1 + m_iS_1).
  \end{align*}

  The first summand on the right hand side has dimension strictly less than $s(2k+1)$ by \Cref{terracini locus of a subspace}.  By counting parameters, the second summand has dimension at most $2s(n-k)$, and the third summand has dimension at most $(r-s)(2n+1)$ (noting that $\ell_im_i$ lies in both $\ell_iS_1$ and $m_iS_1$ for each $i$, so we only count it once).  Putting this together, the dimension is strictly less than $r(2n + 1)$, and the result follows by \Cref{tangential terracini locus - vector space}.
\end{proof}

\Cref{tangential terracini locus - induction via vector spaces} has a partial converse that will prove useful.

\begin{lemma}\label{not in terracini locus in smaller subspace implies not in terracini locus}
  Suppose $d\geq 3$, $P\subset\PP^n$ is a linear space, and $\Gamma$ is a set of $r$ smooth points in $\tau(\nu_d(\PP^n))$ whose base and direction points all lie in $P$.  If $\Gamma\not\in\Ter_r(\tau(\nu_d(P)))$, then $\Gamma\not\in\Ter_r(\tau(\nu_d(\PP^n)))$.
\end{lemma}

\begin{proof}
  Let $V,V'$ be as in the proof of \Cref{tangential terracini locus - induction via vector spaces}, and assume $k < n$, or otherwise we are done.  Then
  \begin{equation*}
    \sum_{i=1}^r\ell_i^{d-2}(\ell_iS_1 + m_iS_1) = \sum_{i=1}^r\ell_i^{d-2}(\ell_iV + m_iV) + \sum_{i=1}^r\ell_i^{d-2}(\ell_iV' + m_iV').
  \end{equation*}
  The first summand $T$ has dimension $r(2k+1)$ by assumption.  Since it involves only variables in $V$, $T$ has only trivial intersection with the second summand, which we denote $U$.

  Now consider $W=\sum_{i=1}^r\ell_i^{d-2}\langle\ell_i,m_i\rangle$.  Note that $\dim W = 2r$, as otherwise there would exist coefficients $a_i,b_i\in\kk$ (not all zero) giving $\sum_{i=1}^r(a_i\ell_i^{d-1}+b_i\ell_i^{d-2}m_i)=0$.  But then $\sum_{i=1}^r(a_i\ell_i^{d-1}v+b_i\ell_i^{d-2}m_iv)=0$ for all $v\in V$, contradicting the assumption that $T$ has full dimension.  Finally, since $W$ involves only the variables in $V$ and $U=WV'$, we see that $\dim U=2r(n-k)$, and the result follows.
\end{proof}

Combining \Cref{tangential secant variety defectivity} and \Cref{tangential terracini locus - induction via vector spaces}, we obtain the following useful result.

\begin{lemma}\label{tangential secant variety defectivity - induction}
  Suppose $d\geq 2$, $P\subset\PP^n$ with $P\cong\PP^k$, $\Gamma$ is a set of $r$ smooth points in $\tau(\nu_d(\PP^n))$, at least $s$ of which lie in $\tau(\nu_d(P))$.  If any of the following are true:
  \begin{enumerate}[(a)]
  \item $d=2$, $k\geq 6$, and $s\in\left\{2,\ldots,\left\lfloor\frac{1}{2k+1}\binom{k+2}{2}\right\rfloor\right\}$,
  \item $(d,k,s)=(3,2,2)$, or
  \item $s>\frac{1}{2k+1}\binom{k+d}{d}$,
  \end{enumerate}
  then $\Gamma\in\Ter_r(\tau(\nu_d(\PP^n)))$.
\end{lemma}

We now look at another source of points on the Terracini locus of a tangential variety of a Veronesean.  Given $\ell,m_1,m_2\in S_1$, note that $[\ell^{d-1}m_1]$ and $[\ell^{d-1}m_2]$, i.e., two points that share a base point, lie in the same tangent space $T_{[\ell^d]}\nu_d(\PP^n)$.  We call these an \textit{intratangent pair}.

\begin{lemma}\label{intratangent pair}
  Suppose $d\geq 2$, $n\geq 1$, $r\geq 2$, and $\Gamma$ is a set of $r$ smooth points in $\tau(\nu_d(\PP^n))$.  If $\Gamma$ contains an intratangent pair, then $\Gamma\in\Ter_r(\tau(\nu_d(\PP^n)))$.
\end{lemma}

\begin{proof}
  We have already proven the result in the case that both points lie on the same tangent line (see \Cref{shared tangent space}), so assume they share a base point but lie on distinct tangent lines.

  Using \Cref{tangential terracini locus - induction via vector spaces}, it is sufficient to prove the result for $n=r=2$.  Under a suitable change of coordinates, we may assume that our intratangent pair is $[y_0^{d-1}y_1]$ and $[y_0^{d-1}y_2]$.  Then $\widehat T_{[y_0^{d-1}y_1]}\tau(\nu_d(\PP^2))$ is spanned by $\{y_0^d, y_0^{d-1}y_1,y_0^{d-1}y_2,y_0^{d-2}y_1^2,\allowbreak y_0^{d-2}y_1y_2\}$ and $\widehat T_{[y_0^{d-1}y_2]}\tau(\nu_d(\PP^2))$ is spanned by $\{y_0^d, y_0^{d-1}y_1,y_0^{d-1}y_2,y_0^{d-2}y_1y_2,\allowbreak y_0^{d-2}y_2^2\}$.  By Grassmann's formula, we see that $\widehat T_{[y_0^{d-1}y_1]}\tau(\nu_d(\PP^2))+\widehat T_{[y_0^{d-1}y_2]}\tau(\nu_d(\PP^2))$ has dimension $5 + 5 - 4 = 6 < 2(2\cdot 2 + 1)$.
\end{proof}

In fact, unless the degree is small, intratangent pairs give us the only points of second Terracini loci.

\begin{lemma}\label{2nd terracini locus}
If $d\geq 5$, then $\Ter_2(\tau(\nu_d(\PP^n)))$ consists of only intratangent pairs.
\end{lemma}

\begin{proof}
  Suppose $\Gamma=\{[\ell_1^{d-1}m_1],[\ell_2^{d-1}m_2]\}$ with $[\ell_1]\neq[\ell_2]$.  Using an appropriate change of coordinates, we may assume that $\ell_1=y_0$ and $\ell_2=y_1$.  Then $\widehat T_{[\ell_1^{d-1}m_1]}\tau(\nu_d(\PP^n))$ is spanned by multiples of $y_0^{d-2}$ and $\widehat T_{[\ell_2^{d-1}m_2]}\tau(\nu_d(\PP^n))$ is spanned by multiples of $y_1^{d-2}$.  These spaces meet trivially, and thus their sum has the expected dimension.
\end{proof}

With the above result, we may wonder what happens for second Terracini loci with smaller degree.  For quadrics, the answer of course is that the Terracini locus is everything by \Cref{tangential secant variety defectivity}.  So we take a moment to examine the cubic and quartic cases.

\begin{lemma}\label{2nd terracini locus - cubics}
  If $n\geq 2$, then $\Ter_2(\tau(\nu_3(\PP^n)))$ consists only of pairs of points whose base and direction points are coplanar.
\end{lemma}

\begin{proof}
  When the base and direction points are coplanar, we have case (b) of \Cref{tangential secant variety defectivity - induction}.  Suppose they are not, and thus $n\geq 3$.  Under a suitable change of coordinates, we may assume that our pair is $\{[y_0^2y_1], [y_2^2y_3]\}$.  Then $\widehat T_{[y_0^2y_1]}\tau(\nu_3(\PP^n))=\langle y_0^2,y_0y_1\rangle S_1$ and  $\widehat T_{[y_2^2y_3]}\tau(\nu_3(\PP^n))=\langle y_2^2,y_2y_3\rangle S_1$, and these spaces intersect trivially.
\end{proof}

Note that the base and direction points of intratangent pairs are necessarily coplanar, and so they are covered in \Cref{2nd terracini locus - cubics} without being explicitly stated.

\begin{lemma}\label{2nd terracini locus - quartics}
  For all positive $n$, $\Ter_2(\tau(\nu_4(\PP^n)))$ consists only of intratangent pairs and pairs of points whose base and direction points are collinear.
\end{lemma}

\begin{proof}
  When the base and direction points are collinear, we have case (c) of \Cref{tangential secant variety defectivity - induction} since $2 > \frac{1}{2\cdot 1 + 1}\binom{1+4}{4}=\frac{5}{3}$.  Suppose they are not, and thus $n\geq 2$, and that we do not have an intratangent pair.  Under a suitable change of coordinates, we may assume that our pair is $\{[y_0^3m_1], [y_1^3m_2]\}$ for some $m_1,m_2\in S_1$.  Then $\widehat T_{[y_0^3m_1]}\tau(\nu_4(\PP^n))=\langle y_0^3,y_0^2m_1\rangle S_1$ and  $\widehat T_{[y_1^3m_2]}\tau(\nu_4(\PP^n))=\langle y_1^3,y_1^2m_2\rangle S_1$, and these spaces might intersect only in $\langle y_0^2y_1^2\rangle$.  But this would require $m_1,m_2\in\langle y_0,y_1\rangle$, contradicting our non-collinearity assumption.
\end{proof}

As we see in the following results, the base points alone can determine whether a set of points lies in the Terracini locus or not.  We begin with a natural generalization of \Cref{2nd terracini locus}.

\begin{lemma}\label{linearly independent base points}
  Suppose $d\geq 5$ and consider a set $\Gamma$ consisting of $r$ smooth points of $\tau(\nu_d(\PP^n))$.  If the base points of $\Gamma$ are linearly independent, then $\Gamma\not\in\Ter_r(\tau(\nu_d(\PP^n)))$.
\end{lemma}

\begin{proof}
  Using an appropriate change of coordinates, we may assume that $\ell_i=y_{i-1}$ for each $i\in\{1,\ldots,r\}$.  Then $\widehat T_{[y_{i-1}^{d-1}m_i]}\tau(\nu_d(\PP^n))\subset y_{i-1}^{d-2}S_2$, and the sum of all such tangent spaces is therefore direct.
\end{proof}

\begin{lemma}\label{terracini locus - base points on a hyperplane}
  Suppose $d\geq 3$, $n\geq 2$, and $r>\frac{1}{2n}\left[\binom{n + d - 2}{d - 1} + \binom{n + d - 1}{d}\right]$.  If $\Gamma\subset\tau(\nu_d(\PP^n))$ consists of $r$ smooth points whose base points all lie on a hyperplane of $\PP^n$, then $\Gamma\in\Ter_r(\tau(\nu_d(\PP^n)))$.
\end{lemma}

\begin{proof}
  Under a suitable change of coordinates, we may assume that for all $[\ell_i^{d-1}m_i]\in\Gamma$, $\ell_i\in\langle y_0,\ldots,y_{n-1}\rangle$.   Then the highest power of $y_n$ that might appear in $V = \sum_{i=1}^r\widehat T_{[\ell_i^{d-1}m_i]}\tau(\nu_d(\PP^n))$ is $y_n^2$ by \Cref{tangent space - tangential}.  Letting $S'=\kk[y_0,\ldots,y_{n-1}]$, we see that $V\subset y_n^2S'_{d-2}\oplus y_nS'_{d-1}\oplus S'_d$.  Since $V\cap y_n^2S'_{d-2}\subset\langle y_n^2\ell_1^{d-2},\ldots,y_n^2\ell_r^{d-2}\rangle$, we have
  \begin{align*}
    \dim V &\leq r + \binom{n + d - 2}{d - 1} + \binom{n + d - 1}{d} \\
           &< r + 2n r = r(2n + 1),
  \end{align*}
  and thus $\Gamma\in\Ter_r(\tau(\nu_d(\PP^n)))$ by definition.
\end{proof}

\begin{example}\label{3 collinear base points 5-2}
  Since $3 > \frac{1}{2\cdot 2}\left[\binom{2 + 5 - 2}{5 - 1} + \binom{2 + 5 - 1}{5}\right]$, we see by \Cref{terracini locus - base points on a hyperplane} that if $\Gamma$ consists of 3 points of $\tau(\nu_5(\PP^2))$ with collinear base points, then $\Gamma\in\Ter_3(\tau(\nu_5(\PP^2)))$.
\end{example}

\begin{remark}
  \Cref{terracini locus - base points on a hyperplane} provides a simple proof of \cite[Proposition 3.2(i)]{CGG02}.  Indeed, when  $(d,n,r)=(3,2,2)$, the inequality is satisfied.  And since any pair of base points will span a line, the Terracini locus consists of all possible pairs of smooth points, i.e., $\sigma_2(\tau(\nu_3(\PP^2)))$ is defective.  Furthermore, since any three points lie on a plane in $\PP^3$ and $3 + \binom{3 + 3 - 2}{3 - 1} + \binom{3 + 3 - 1}{3}=19 < 20 = \min\left\{3(2\cdot 3 + 1), \binom{3 + 3}{3}\right\}$, the proof also gives us $\sigma_3(\tau(\nu_3(\PP^3)))$ defective, and similarly, $\sigma_4(\tau(\nu_3(\PP^4)))$, recovering \cite[Proposition 3.4]{CGG02}.
\end{remark}

\begin{lemma}\label{terracini locus - veronese terracini base points}
  If $d\geq 2$, $\ell_1,\ldots,\ell_r\in S_1$, and $\{[\ell_1^d],\ldots,[\ell_r^d]\}\in\Ter_r(\nu_d(\PP^n))$, then for any $m_1,\ldots,m_r\in S_1$ for which $\Gamma=\{[\ell_1^{d-1}m_1],\ldots,[\ell_r^{d-1}m_r]\}\in\tau(\nu_d(\PP^n))_{sm}^{(r)}$, we also have $\Gamma\in\Ter_r(\tau(\nu_d(\PP^n)))$.
\end{lemma}

\begin{proof}
  For each $i$, let $S_1^{(i)}\subset S_1$ be an $n$-dimensional subspace satisfying $\langle \ell_i\rangle\oplus S_1^{(i)}=S_1$ so that $\ell_iS_1+m_iS_1 = \ell_iS_1 + m_iS_1^{(i)}$.  Then by \Cref{tangent space - chow-veronese},
  \begin{align*}
    \sum_{i=1}^r\widehat T_{[\ell_i^{d-1}m_i]}\tau(\nu_d(\PP^n)) &= \sum_{i=1}^r\ell_i^{d-2}\left(\ell_iS_1 + m_iS_1^{(i)}\right) \\
                                                                 &= \sum_{i=1}^r\widehat T_{[\ell_i^d]}\nu_d(\PP^n) +  \sum_{i=1}^r\ell_i^{d-2}m_iS_1^{(i)}.
  \end{align*}
  Since $\{[\ell_1^d],\ldots,[\ell_r^d]\}\in\Ter_r(\nu_d(\PP^n))$, the first summand has dimension strictly less than $r(n+1)$.  Since the second summand has dimension at most $rn$, the result follows.
\end{proof}

\begin{example}\label{3 collinear base points 4-3}
  Suppose $\Gamma\subset\tau(\nu_4(\PP^3))$ consists of 3 points with collinear base points.  The fourth powers of these base points form an element of $\Ter_3(\nu_4(\PP^3))$ by \cite[Theorem 6.6]{GSTT25}, and so by \Cref{terracini locus - veronese terracini base points}, $\Gamma\in\Ter_3(\tau(\nu_4(\PP^3)))$.
\end{example}

\begin{remark}
  \Cref{terracini locus - base points on a hyperplane,terracini locus - veronese terracini base points} each stands on its own.  Indeed, for \Cref{3 collinear base points 5-2}, \Cref{terracini locus - veronese terracini base points} does not help since $3<\frac{5 + 2}{2}$, and so $\Ter_3(\nu_5(\PP^2))=\emptyset$ by \cite[Theorem 1.1]{LM24}.  And for \Cref{3 collinear base points 4-3}, \Cref{terracini locus - base points on a hyperplane} does not help since $3\leq\frac{1}{2\cdot 3}\left[\binom{3+4-2}{4-1} + \binom{3+4-1}{4}\right]$.
\end{remark}

We close this section with a useful remark.

\begin{remark}\label{distinct points}
  If we would like to show that $\Gamma=\{[\ell_1^{d-1}m_1],\ldots,[\ell_r^{d-1}m_r]\}$ lies outside of $\Ter_r(\tau(\nu_d(\PP^n)))$, then we may assume that the $[\ell_i]$ and $[m_i]$ are $2r$ distinct points of $\PP^n$.  Indeed, if $[\ell_i]=[\ell_j]$ for any $i\neq j$, then we have an intratangent pair and $\Gamma$ certainly lies in the Terracini locus by \Cref{intratangent pair}.  If $[m_i]=[\ell_i]$ for some $i$, then we have a singular point.  And if $[m_i]=[\ell_j]$ or $[m_i]=[m_j]$ for some $i\neq j$, then we may simply choose another point on the line $\langle[\ell_i],[m_i]\rangle$ to replace $[m_i]$ by \Cref{shared tangent space}.
\end{remark}

\section{Inverse systems and $(2,3)$-point schemes}
\label{inverse system section}

We now will explore a beautiful connection between tangential varieties of Veroneseans and a particular family of fat point schemes first established in \cite{CGG02}.

Consider the polynomial ring $R=\kk[x_0,\ldots,x_n]$.  If we identify $x_i$ with $\frac{\partial}{\partial y_i}$, then each element of $R$ corresponds to a differential operator on $S$, which gives us the \textit{apolarity} action of $R$ on $S$.  Furthermore, given an ideal $I\subset R$, its \textit{inverse system} is the graded $R$-submodule of $S$ given by
\begin{equation*}
  I^{-1} = \{g\in S:f\cdot g = 0\:\:\forall f\in I\}.
\end{equation*}
See \cite[Lecture 2]{G96} for an excellent introduction to the theory of inverse systems.

We will also use $R$ as the homogeneous coordinate ring for $\PP^n$.  So in particular, the points in $\PP^n=\PP(S_1)$ will correspond to equivalence classes of linear forms in $S$ but the subschemes of $\PP^n$ will be cut out by homogeneous polynomials in $R$.

The following result fills in the details of an omitted ``easy computation'' from \cite{CGG02}.

\begin{lemma}\label{tangent space as inverse system}
  Suppose $n\geq 2$ and $\ell,m\in S_1$.  If $\pp\subset R$ is the ideal of the point $[\ell]\in\PP^n$ and $\LL\subset R$ the ideal of the line in $\PP^n$ through $[\ell]$ and $[m]$, then for all $d\geq 2$,
  \begin{equation*}
    \widehat T_{[\ell^{d-1}m]}\tau(\nu_d(\PP^n)) = (\LL^2 + \pp^3)^{-1}_d.
  \end{equation*}
\end{lemma}

\begin{proof}
  Under an appropriate change of coordinates, we may assume that $\ell=y_0$ and $m=y_1$.  Then $\pp = (x_1,\ldots,x_n)$ and $\LL=(x_2,\ldots,x_n)$, and so $(\LL^2 + \pp^3)_d$ is spanned by monomials of the form $f=x_0^{d_0}\cdots x_n^{d_n}$ where
  \begin{enumerate}[(a)]
  \item $\frac{f}{x_0^{d_0}x_1^{d_1}}\in\LL^2$, i.e., $d_0+d_1\leq d-2$, or
  \item $\frac{f}{x_0^{d_0}}\in\pp^3$, i.e., $d_0\leq d - 3$.
  \end{enumerate}

  It follows by \cite[Remark 3, p. 17]{G96} that $(\LL^2 + \pp^3)^{-1}_d$ is spanned by monomials of the form $y_0^{d_0}\cdots y_n^{d_n}$, where $d_0+d_1\geq d - 1$ and $d_0\geq d - 2$.  So either
  \begin{enumerate}[(a)]
  \item $d_0=d-2$ and $d_1\geq 1$, giving us $y_0^{d-2}y_1y_i$ for any $i$, or
  \item $d_0\geq d-1$, giving us $y_0^{d-1}y_i$ for any $i$,
  \end{enumerate}
  which is exactly $\widehat T_{[y_0^{d-1}y_1]}\tau(\nu_d(\PP^n))$ by \Cref{tangent space - tangential}.
\end{proof}

This observation leads us to the following definition.

\begin{definition}\label{(2,3)-point definition}
  Suppose $p\in\PP^n$ is a reduced point and $L\ni p$ is a line.  If $\pp,\LL\subset R$ are their respective ideals, then the zero-dimensional scheme supported at $p$ and cut out by $\LL^2 + \pp^3$, i.e., the intersection of a double line and a triple point on that line, is a \textit{$(2,3)$-point}.
\end{definition}

By \cite[Proposition 2.6]{G96}, $I^{-1}+J^{-1}=(I\cap J)^{-1}$ for any ideals $I$ and $J$.  So together with \Cref{tangent space as inverse system}, we have a dictionary between sums of tangent spaces of $\tau(\nu_d(\PP^n))$ and $(2,3)$-point schemes.  In particular, for any set $\Gamma=\{[\ell_1^{d-1}m_1],\ldots,[\ell_r^{d-1}m_r]\}$ of smooth points in $\tau(\nu_d(\PP^n))$, we define the ideal
\begin{equation*}
  I_\Gamma = \bigcap_{i=1}^r\left(\LL_i^2 + \pp_i^3\right)
\end{equation*}
and the zero-dimensional scheme $Z_\Gamma = V(I_\Gamma)\subset\PP^n$, where $\pp_i$ is the ideal of $[\ell_i]$ and $\LL_i$ the ideal of $\langle[\ell_i],[m_i]\rangle$ for each $i$.

We can now describe the Terracini loci of $\tau(\nu_d(\PP^n))$ using these $(2,3)$-point schemes.  We restrict ourselves to the case where the $(2,3)$-points of $Z_\Gamma$ are supported at distinct points of $\PP^n$, which corresponds to $\Gamma$ containing no intratangent pairs.

\begin{lemma}\label{tangential terracini locus - sheaf cohomology}
  If $d,n\geq 2$ and $\Gamma$ is a set of $r$ smooth points in $\tau(\nu_d(\PP^n))$ containing no intratangent pairs, then $\Gamma\in\Ter_r(\tau(\nu_d(\PP^n)))$ if and only if
  \begin{equation*}
    h^1(\PP^n,\II_{Z_\Gamma}(d)) > 0.
  \end{equation*}
\end{lemma}

\begin{proof}
  We start with the following standard fundamental short exact sequence of sheaves:
\begin{equation*}
  0 \to \II_{Z_\Gamma} \to \OO_{\PP^n} \to \OO_{Z_\Gamma} \to 0
\end{equation*}
Twisting and taking cohomology, we have
\begin{align*}
  0 &\to H^0(\PP^n,\II_{Z_\Gamma}(d)) \to H^0(\PP^n,\OO_{\PP^n}(d)) \to H^0(\PP^n,\OO_{Z_\Gamma}(d)) \\
    & \to H^1(\PP^n,\II_{Z_\Gamma}(d)) \to H^1(\PP^n,\OO_{\PP^n}(d)) \to \cdots
\end{align*}

Note the following observations.
\begin{enumerate}[(a)]
\item The ideal of each $(2,3)$-point in $Z_\Gamma$ is saturated.  Indeed, consider the coordinate change from \Cref{tangent space as inverse system}, where our ideal is generated in $x_1,\ldots,x_n$.  Since it is missing $x_0$, saturating out the irrelevant ideal does nothing.  And since intersections preserve saturatedness, $I_\Gamma$ is itself saturated.  It follows that $I_\Gamma = H^0_*(\PP^n,\II_{Z_\Gamma})$, and so $h^0(\PP^n,\II_{Z_\Gamma}(d)) = h_{I_\Gamma}(d)$, the value at $d$ of the Hilbert function of $I_\Gamma$.
\item By \cite[Remark 1, p. 17]{G96}, we have $h_{I_\Gamma}(d) + h_{I_\Gamma^{-1}}(d) = h_R(d)$, or $h_{I_\Gamma}(d)=\binom{n+d}{d}-h_{I_\Gamma^{-1}}(d)$.
\item By \cite[Theorem III.5.1]{H77}, we have $h^0(\PP^n,\OO_{\PP^n}(d))=\binom{n+d}{d}$ and $h^1(\PP^n,\OO_{\PP^n}(d))=0$.
\item Since $Z_\Gamma$ is zero-dimensional, we know that $h^0(\PP^n,\OO_{Z_\Gamma}(d))=\deg Z_\Gamma$.  To find this, we first find the degree of a single $(2,3)$-point.  We may again fall back on the change of coordinates from \Cref{tangent space as inverse system}.  It suffices to count the degree $d$ monomials not in $(x_2,\ldots,x_n)^2+(x_1,\ldots,x_n)^3$, which are
  \begin{equation*}
    \underbrace{x_0^d}_1,\underbrace{x_0^{d-1}x_1,\ldots,x_0^{d-1}x_n}_n,\underbrace{x_0^{d-2}x_1^2,\ldots,x_0^{d-2}x_1x_n}_n,
  \end{equation*}
  and so the degree of each $(2,3)$-point is $2n+1$.  It follows that $\deg Z_\Gamma = r(2n+1)$ since their supports are distinct.
\end{enumerate}

Putting this all together, we have

\begin{gather*}
  \underbrace{\binom{n+d}{d}-h_{I_\Gamma^{-1}}(d)}_{h^0(\PP^n,\II_{Z_\Gamma}(d))} - \underbrace{\binom{n+d}{d}}_{h^0(\PP^n,\OO_{\PP^n}(d))} + \underbrace{r(2n+1)}_{h^0(\PP^n,\OO_{Z_\Gamma}(d))} - h^1(\PP^n,\II_{Z_\Gamma}(d)) = 0\text{, or} \\
   h_{I_\Gamma^{-1}}(d) =  r(2n+1) -h^1(\PP^n,\II_{Z_\Gamma}(d)).
\end{gather*}

However, by \Cref{tangent space as inverse system} and \cite[Proposition 2.6]{G96}, this is just
  \begin{equation*}
    \dim\sum_{i=1}^r\widehat T_{[\ell_i^{d-1}m_i]}\tau(\nu_d(\PP^n)) = r(2n+1) -h^1(\PP^n,\II_{Z_\Gamma}(d)),
  \end{equation*}
  and so the result follows by \Cref{tangential terracini locus - vector space}.
\end{proof}

Suppose $\Gamma$ has no intratangent pairs.  Note that $H^0(\PP^n,\II_{Z_\Gamma}(d))=(I_\Gamma)_d$ is the $\kk$-vector space of $d$-ics that vanish on $Z_\Gamma$, and by the above argument we have
\begin{equation*}
  h_{I_\Gamma}(d) = \binom{n+d}{d} - r(2n+1) + h^1(\PP^n,\II_{Z_\Gamma}(d)).
\end{equation*}
So in the language of interpolation, when $\Gamma\not\in\Ter_r(\tau(\nu_d(\PP^n)))$, we have the expected dimension $h_{I_\Gamma}(d) = \binom{n+d}{d} - r(2n+1)$ and say that $Z_\Gamma$ ``imposes independent conditions'' on $d$-ics.  However, if $\Gamma$ is in the Terracini locus, then the first cohomology is nonzero and we have more $d$-ics vanishing on $Z_\Gamma$ than we expected (provided that $r$ is subabundant).  In this case, $Z_\Gamma$ ``fails to impose independent conditions'' on $d$-ics.

\section{Interlude: Binary forms}\label{sec: binary forms}

Notice that in the previous section, we omitted the binary form case where $n=1$, since otherwise we could not define a $(2,3)$-point scheme.  Let us take a moment to remedy this situation and describe the Terracini loci of tangential varieties of rational normal curves.

We begin with an analogue of \Cref{tangent space as inverse system}.

\begin{lemma}\label{tangent space as inverse system - binary forms}
  Suppose $n=1$ and $\ell,m\in S_1$.  If $\pp\subset R$ is the ideal of the point $[\ell]\in\PP^1$, then for all $d\geq 2$,
  \begin{equation*}
    \widehat T_{[\ell^{d-1}m]}\tau(\nu_d(\PP^1))=(\pp^3)_d^{-1}.
  \end{equation*}
\end{lemma}

\begin{proof}
  Under a suitable change of coordinates, we may assume that $\ell=y_0$ and $m=y_1$.  Then $\pp^3=(x_1^3)$, and so $(\pp^3)_d^{-1}$ is spanned by all monomials of the form $y_0^{d_0}y_1^{d_1}$ with $d_1<3$, i.e., $\{y_0^d, y_0^{d-1}y_1,y_0^{d-2}y_1^2\}$.  But this is exactly $\widehat T_{[y_0^{d-1}y_1]}\tau(\nu_d(\PP^1))$ by \Cref{tangent space - tangential}.
\end{proof}

So we see that just as points in $\tau(\nu_d(\PP^n))$ correspond to $(2,3)$-point schemes in $\PP^n$ when $n\geq 2$, points in $\tau(\nu_d(\PP^1))$ correspond to triple point schemes in $\PP^1$.

Now consider a set $\{[\ell_1^{d-1}m_1],\ldots,[\ell_r^{d-1}m_r]\}$ of smooth points in $\tau(\nu_d(\PP^1))$ and let $\pp_i\subset R$ be the ideal of $[\ell_i]$.  Then $\pp_i=(h_i)$ for some linear form $h_i$, and so if $I$ is the ideal of our triple point scheme, then we have
\begin{equation*}
  I=\bigcap_{i=1}^r\pp_i^3 = \bigcap_{i=1}^r(h_i^3)=(\lcm\{h_1^3,\ldots,h_r^3\}).
\end{equation*}
But $\lcm\{h_1^3,\ldots,h_r^3\}$ is the product of the cubes of the distinct (up to scalar multiplication) $h_i$.  Without loss of generality, suppose $h_1,\ldots,h_s$ are distinct for some $s\leq r$, and so $I=(h_1^3\cdots h_s^3)$.  Note then that
\begin{equation*}
  I_d=\begin{cases}
    \{h_1^3\cdots h_s^3f:f\in R_{d-3s}\}&\text{if }d\geq 3s \\
    0 &\text{otherwise,}
    \end{cases}
  \end{equation*}
  and so $h_I(d)=\max\{d - 3s + 1, 0\}$.  It follows then by \Cref{tangent space as inverse system - binary forms} and \cite[Proposition 2.6]{G96} that
  \begin{equation*}
    \dim\sum_{i=1}^r\widehat T_{[\ell_i^{d-1}m_i]}\tau(\nu_d(\PP^1))=\min\{3s, d + 1\},
  \end{equation*}
  which is the expected dimension in the subabundant case only when $s=r$, i.e., all of the $h_i$ (and thus the $\ell_i$, since $h_i=ax_0+bx_1$ implies $\ell_i=by_0-ay_1$, or a scalar multiple thereof) are distinct up to scalar multiplication.

  If they are not distinct, say $\ell_i$ and $\ell_j$ are scalar multiples of one another, then $[\ell_i^{d-1}m_i]$ and $[\ell_j^{d-1}m_j]$ are two points on the same tangent line to the rational normal curve at $[\ell_i^d]$, i.e., they form an intratangent pair.

  We now summarize what this means for Terracini loci in the following generalization of \cite[Proposition 3.1]{CGG02}.

  \begin{theorem}\label{tangential terracini locus - binary forms}
    A set $\Gamma$ of $r$ smooth points in $\tau(\nu_d(\PP^1))$ lies in $\Ter_r(\tau(\nu_d(\PP^1)))$ if and only if
    \begin{enumerate}[(a)]
    \item $r>\frac{d+1}{3}$ or
    \item $\Gamma$ contains an intratangent pair.
    \end{enumerate}
  \end{theorem}

\section{The Horace lemma}\label{sec: horace}

We now introduce the widely-used \textit{Horace lemma} \cite{H81,H85} that we will use to study the Terracini loci of tangential varieties of Veroneseans.  It gets its name from the Roman legend of Publius Horatius, who cunningly separated three brothers from the rival city of Alba Longa so that he could kill them one-by-one.  Instead of slaying three brothers, we will instead compute cohomologies for the ideal sheaves of three schemes.

Consider a projective scheme $Z$ and a hyperplane $H$, both in $\PP^n$.  If $I$ is the saturated ideal of $Z$ and $H$ is cut out by the linear form $h\in R_1$, then there exists an injective map $(I:(h))(-1)\to I$ of graded $R$-modules defined by multiplication by $h$.  The image of this map is $I\cap(h)$, and by the second isomorphism theorem, its cokernel is isomorphic to $(I + (h))/(h)$.  This gives rise to the following exact sequence of graded $R$-modules.
\begin{equation*}
  0 \to (I:(h))(-1) \to I \to (I + (h))/(h) \to 0
\end{equation*}

The scheme defined by $I:(h)$ is known as the \textit{residual} of $Z$ with respect to $H$, denoted $\Res_H Z$.  This is the scheme-theoretic set difference, although set-theoretically, we will see that some points in the support of $Z$ that also lie in $H$ may remain in the support of the residual.  Furthermore, the scheme defined by $(I + (h))/(h)$ is the \textit{trace}, denoted $\Tr_H Z$.  This is the scheme-theoretic intersection of $Z$ and $H$ viewed as a subscheme of $H$.

So sheafifying and twisting by $d$, we obtain the \textit{Castelnuovo exact sequence}.

\begin{equation*}
  0 \to \II_{\Res_HZ}(d - 1) \to \II_Z(d) \to \II_{\Tr_H Z}(d) \to 0
\end{equation*}

This leads directly to the Horace lemma.

\begin{lemma}[Horace lemma]\label{horace lemma}
  Suppose $Z$ is a closed subscheme of $\PP^n$ and $H$ is a hyperplane.  If $H^i(\PP^n,\II_{\Res_HZ}(d-1))=H^i(H,\II_{\Tr_HZ}(d))=0$ for some $d\in\ZZ$ and $i\in\ZZ_{\geq 0}$, then $H^i(\PP^n,\II_Z(d))=0$.
\end{lemma}

\begin{proof}
  Taking cohomology, the Castelnuovo exact sequence leads to the following exact sequence, and the result follows immediately.
  \begin{equation*}
    H^i(\PP^n,\II_{\Res_HZ}(d - 1)) \to H^i(\PP^n,\II_Z(d)) \to H^i(H,\II_{\Tr_H Z}(d))\qedhere
  \end{equation*}
\end{proof}

In the literature, we often see the statement $H^i(\PP^n,\II_{\Res_HZ}(d-1))=0$ referred to as \textit{degue} and $H^i(H,\II_{\Tr_HZ}(d))=0$ as \textit{dime}.  This gives us a powerful inductive technique.  In practice, we want to show that cohomologies of the twists of the ideal sheaf of a particular scheme vanish, but it is often simpler to verify the degue and dime statements.  These statements depend on schemes in lower degree and dimension, respectively (hence the names).

Now recall our notion of a $(2,3)$-point, i.e., a fat point scheme with ideal $\LL^2 + \pp^3$, where $\LL$ is the ideal of a line $L$ through a point $p$ with ideal $\pp$.  Related fat point schemes with the same support are the 2-jet (see \cite{AH97}) with ideal $\LL + \pp^2$ and the double point with ideal $\pp^2$.  As we will see, both will appear when we start finding residuals and traces of $(2,3)$-point schemes.  As with $(2,3)$-points, it only makes sense to consider schemes of 2-jets when $n\geq 2$.

Now consider a hyperplane $H$ through $p$ and let $Z$ be a 2-jet or $(2,3)$-point supported at $p$.  We say that $Z$ is \textit{partially supported} in $H$ if $L\cap H = p$ and \textit{fully supported} in $H$ if $L\subset H$.

\begin{lemma}\label{residuals and traces}
  Suppose $n\geq 2$ and $H$ is a hyperplane.
  \begin{enumerate}[(a)]
  \item If $Z$ is a $(2,3)$-point partially supported in $H$, then $\Res_HZ$ is a double point.
  \item If $Z$ is a $(2,3)$-point partially supported in $H$, then $\Tr_HZ$ is a double point.
  \item If $Z$ is a $(2,3)$-point fully supported in $H$, then $\Res_HZ$ is a 2-jet.
  \item If $Z$ is a $(2,3)$-point fully supported in $H$, then $\Tr_HZ$ is a $(2,3)$-point if $n\geq 3$ or a triple point if $n = 2$.
  \item If $Z$ is a double point supported in $H$, then $\Res_HZ$ is a reduced point.
  \item If $Z$ is a double point supported in $H$, then $\Tr_HZ$ is a double point.    
  \item If $Z$ is a 2-jet fully supported in $H$, then $\Res_HZ$ is empty.
  \item If $Z$ is a 2-jet fully supported in $H$, then $\Tr_HZ$ is a 2-jet if $n\geq 3$ or a double point if $n=2$.
  \end{enumerate}
\end{lemma}

\begin{proof}
  These are straightforward computations using our usual change of coordinates $\LL=(x_2,\ldots,x_n)$ and $\pp=(x_1,\ldots,x_n)$, with either $H=(x_1)$ (partially supported) or $H=(x_n)$ (fully supported), followed by \cite[\S 4.4, Theorem 14]{CLO07} to compute the colon ideal when finding the residual.

  For example, for (g), $\LL + \pp^2 = (x_1^2,x_2,\ldots,x_n)$, so $(\LL + \pp^2)\cap(x_n) = (x_n)$, and thus $(\LL + \pp^2):(x_n) = R$.  Cases (a)--(d) are also found in \cite[Remark 1.7]{BCGI09}.
\end{proof}

\begin{remark}\label{inductive residual}
We see then that if we apply $\Res_H$ enough times on a $(2,3)$-point scheme, the subscheme supported in $H$ will eventually vanish.  We denote by $\Res_H^kZ$ the scheme $\underbrace{\Res_H(\cdots(\Res_H}_{k\text{ times}}Z)\cdots)$.

Note in particular that if $H$ is a hyperplane and $Z$ is a scheme consisting of $r$ $(2,3)$-points, $s'$ of which are partially supported in $H$ and exactly $s''$ of which are fully supported in $H$, then, letting $s=s'+s''$,
\begin{enumerate}[(a)]
\item $\Res_HZ$ consists of $r-s$ $(2,3)$-points supported away from $H$, $s'$ double points supported in $H$, and $s''$ 2-jets supported in $H$,
\item $\Tr_HZ$ consists of $s'$ double points and $s''$ $(2,3)$-points (if $n\geq 3$) or triple points (if $n=2$) as a subscheme of $H$,
\item $\Res_H^2Z$ consists of $r-s$ $(2,3)$-points supported away from $H$ and $s'$ reduced points on $H$,
\item $\Tr_H(\Res_HZ)$ consists of $s'$ double points and $s''$ 2-jets (if $n\geq 3$) or double points (if $n=2$) as a subscheme of $H$,
\item $\Res_H^3Z$ consists of $r - s$ $(2,3)$-points supported away from $H$, and
\item $\Tr_H(\Res_H^2Z)$ consists of $s'$ reduced points as a subscheme of $H$.
\end{enumerate}
\end{remark}

Our goal is to develop an inductive technique similar to the one used in \cite[Lemma 6.1]{GSTT25} for Terracini loci of Veronese varieties, but applied to their tangential varieties.  We first establish the following technical definition.

\begin{definition}\label{horace general}
  A set $\Gamma\subset\tau(\nu_d(\PP^n))$ of smooth points with distinct base points that all lie in a hyperplane $H\subset\PP^n$ is \textit{Horace general} in $H$ if the following are all true:
  \begin{enumerate}[(a)]
  \item $h^1(H, \II_{\Tr_HZ_\Gamma}(d))=0$,
  \item $h^1(H, \II_{\Tr_H(\Res_HZ_\Gamma)}(d - 1)) = 0$,
  \item $h^1(H, \II_{\Tr_H(\Res_H^2Z_\Gamma)}(d - 2)) = 0$.
  \end{enumerate}
\end{definition}

By \Cref{inductive residual}, we see that we can determine Horace generality by looking at the following:
\begin{enumerate}[(a)]
\item a scheme of $s'$ double and $s''$ $(2,3)$-points or triple points, which corresponds via apolarity to the join $\sigma_{s',s''}(\nu_d(\PP^{n-1}))$ of \cite{B22},
\item a scheme of $s'$ double points and $s''$ 2-jets or double points, and
\item a scheme of $s'$ reduced points, about which much is known, e.g., \cite{EGH96}.
\end{enumerate}

\begin{lemma}\label{horace induction}
  Suppose $d\geq 5$, $n\geq 2$, and consider a set $\Gamma\subset\tau(\nu_d(\PP^n))$ consisting of $r$ smooth points containing no intratangent pairs.  Furthermore, for a given hyperplane $H\subset\PP^n$, let $\Gamma'=\{[\ell^{d-1}m]\in\Gamma:[\ell]\in H\}$ and $s = |\Gamma'|$.  If $\Gamma'$ is Horace general in $H$ and $\{[\ell^{d-4}m]:[\ell^{d-1}m]\in\Gamma\setminus\Gamma'\}\not\in\Ter_{r-s}(\tau(\nu_{d-3}(\PP^n)))$, then $\Gamma\not\in\Ter_r(\tau(\nu_d(\PP^n)))$.
\end{lemma}

\begin{proof}
  This follows from \Cref{tangential terracini locus - sheaf cohomology} by applying \Cref{horace lemma} three times, as suggested by \Cref{inductive residual}.  See \Cref{visualization of horace induction} for a visualization.
\end{proof}

  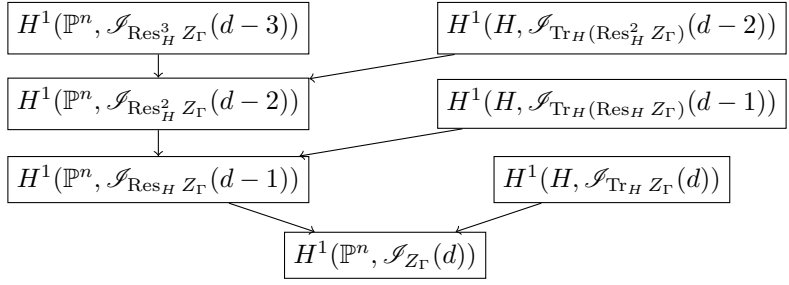
\begin{figure}
    \begin{tikzpicture}
      \node (A) at (0,0) [draw] {$H^1(\PP^n,\II_{Z_\Gamma}(d))$};
      \node (B) at (-3,1) [draw] {$H^1(\PP^n,\II_{\Res_HZ_\Gamma}(d-1))$};
      \node (C) at (3,1) [draw] {$H^1(H,\II_{\Tr_HZ_\Gamma}(d))$};
      \node (D) at (-3,2) [draw] {$H^1(\PP^n,\II_{\Res_H^2Z_\Gamma}(d-2))$};
      \node (E) at (3,2) [draw] {$H^1(H,\II_{\Tr_H(\Res_HZ_\Gamma)}(d-1))$};
      \node (F) at (-3,3) [draw] {$H^1(\PP^n,\II_{\Res_H^3Z_\Gamma}(d-3))$};
      \node (G) at (3,3) [draw] {$H^1(H,\II_{\Tr_H(\Res_H^2Z_\Gamma)}(d-2))$};
      \draw[->] (B) -- (A);
      \draw[->] (C) -- (A);
      \draw[->] (D) -- (B);
      \draw[->] (E) -- (B);
      \draw[->] (F) -- (D);
      \draw[->] (G) -- (D);
    \end{tikzpicture}
    \caption{Applying the Horace lemma in the proof of \Cref{horace induction}}
    \label{visualization of horace induction}
  \end{figure}

  For a given set $\Gamma=\{[\ell_1^{d-1}m_1],\ldots,[\ell_r^{d-1}m_r]\}$ of smooth points of $\tau(\nu_d(\PP^n))$, we denote by $\widetilde\Gamma$ the set $\{[\ell_1],[m_1],\ldots,[\ell_r],[m_r]\}\subset\PP^n$.  We may assume that $|\widetilde\Gamma|=2r$ by \Cref{distinct points}.

  We now generalize \Cref{tangential terracini locus - binary forms} to an arbitrary number of variables.

\begin{theorem}\label{no intratangent pairs implies not in terracini locus}
  Suppose $d\geq 5$, $n$ is a positive integer, and $\Gamma$ is a set of $r\leq\frac{d+1}{3}$ smooth points in $\tau(\nu_d(\PP^n))$. Then $\Gamma\in\Ter_r(\tau(\nu_d(\PP^n)))$ if and only if it contains an intratangent pair.
\end{theorem}

\begin{proof}
  The $n=1$ case is \Cref{tangential terracini locus - binary forms}, so we may assume $n\geq 2$.  Furthermore, the reverse direction is \Cref{intratangent pair}, so we need to only prove the (contrapositive of the) forward direction.

  By \Cref{tangential terracini locus - induction via vector spaces}, it is sufficient to prove the result for $r=\left\lfloor\frac{d+1}{3}\right\rfloor$.  We proceed by 3-step induction on $d$.  For the base cases, we use $d\in\{5,6,7\}$ (and so $r=2$), for which the result is true by \Cref{2nd terracini locus}.

  Suppose $d\geq 8$.  Pick a hyperplane $H\subset\PP^n$ containing $[\ell_1]$ but no other point of $\widetilde\Gamma$.  Since it contains a single point, $\{[\ell_1^{d-1}m_1]\}$ is certainly Horace general.  By induction, $\{[\ell_2^{d-4}m_2],\ldots,[\ell_r^{d-4}m_r]\}\not\in\Ter_{\left\lfloor\frac{d-3+1}{3}\right\rfloor}(\tau(\nu_{d-3}(\PP^n)))$, and so the result follows by \Cref{horace induction}.
\end{proof}

We summarize the cases for which we can completely describe the Terracini locus of a tangential variety of a Veronesean in \Cref{complete descriptions}.

\begin{table}
  \begin{tabular}{|c|c|c|c|c|}
    \hline
    $d$ & $n$ & $r$ & $\Ter_r(\tau(\nu_d(\PP^n)))$ & Reference \\ \hline\hline
    2 & $6,\ldots$ & $2,\ldots,\left\lfloor\frac{1}{2n+1}\binom{n+2}{2}\right\rfloor$ & \multirow{3}{*}{$\tau(\nu_d(\PP^n))_{sm}^{(r)}$} & \multirow{3}{*}{\Cref{tangential secant variety defectivity}} \\
    3 & 2 & 2 & & \\
    $2,\ldots$ & $1,\ldots$ & $\left\lfloor\frac{1}{2n+1}\binom{n+d}{d}\right\rfloor + 1,\ldots$ & & \\ \hline
    3 & $3,\ldots$ & 2 & $\{\Gamma:\langle\widetilde\Gamma\rangle\cong\PP^k,k\leq 2\}$ & \Cref{2nd terracini locus - cubics} \\ \hline
    4 & $2,\ldots$ & 2 &  $\{\Gamma:\langle\widetilde\Gamma\rangle\cong\PP^1\text{ or intratangent pair}\subset\Gamma\}$ & \Cref{2nd terracini locus - quartics} \\ \hline
    $5,\ldots$ & $1,\ldots$ & $2,\ldots,\left\lfloor\frac{d+1}{3}\right\rfloor$ & $\{\Gamma:\text{intratangent pair}\subset\Gamma\}$ & \Cref{no intratangent pairs implies not in terracini locus} \\ \hline
  \end{tabular}
  \caption{Complete descriptions of Terracini loci of $\tau(\nu_d(\PP^n))$}
  \label{complete descriptions}
\end{table}

\section{Moment varieties and local mixtures of location families}\label{sec: moment varieties}
We now turn to an application of our work in algebraic statistics.  Since statisticians work over $\RR$ and we have been working over an algebraically closed field, we will assume for the remainder of this paper that $\kk=\overline\RR=\CC$.

One of the most classical point estimation techniques in statistics is the \textit{method of moments}.  Consider a sample space $\Omega$ with $\sigma$-algebra $\mathcal F$, a parameter space $\Theta\subset\RR^k$ with $\relint\Theta\neq\emptyset$\footnote{Here, $\relint\Theta$ is the \textit{relative interior} of $\Theta$, which is its interior within the affine hull $\aff\Theta$ of $\Theta$ in $\RR^k$.  Since it is nonempty, the Zariski closure $\overline\Theta$ is an affine subspace of $\CC^k$.}, and a family $\mathcal P=\{P_{\vb\theta}:\mathcal F\to[0,1]\}_{\vb\theta\in\Theta}$ of probability distributions.  For a random vector $\vb X=(X_1,\ldots,X_n):\Omega\to\RR^n$  and any vector $\vb \alpha=(\alpha_1,\ldots,\alpha_n)$ of nonnegative integers, we define the $\vb\alpha$th \textit{moment} of $\vb X$ with respect to $P_{\vb\theta}$ as
\begin{equation*}
  E_{\vb\theta}(\vb X^{\vb\alpha}) = \int_{\Omega} \vb X(\omega)^{\vb\alpha}dP_{\vb\theta}(\omega),
\end{equation*}
assuming this integral exists.  Note that $\vb X^{\vb\alpha}$ denotes $X_1^{\alpha_1}\cdots X_n^{\alpha_n}$.  Note also that $E_{\vb\theta}(\vb X^{\vb 0})=1$ for all $\vb\theta$.

\begin{definition}\label{def: moment variety}
  Consider a positive integer $d$, and let $\mathcal P$ be a family of probability distributions with parameter space $\Theta$.  If the map $\vb\theta\mapsto (E_{\vb\theta}(\vb X^{\vb\alpha}))_{|\vb\alpha|\leq d}$ extends to a rational map $\vb m_{\mathcal P,d}:\overline\Theta\dashrightarrow\CC^{\binom{n+d}{d}}$, then
  \begin{equation*}
    \mathcal M_{\mathcal P,d} = \overline{\{[\vb m_{\mathcal P,d}(\vb\theta)]:\vb\theta\in\Theta\}}\subset\PP^{\binom{n+d}{d}-1}
  \end{equation*}
  is the $d$th \textit{moment variety} of $\mathcal P$.
\end{definition}

We will use the shortened notation $\vb m_d$ and $\mathcal M_d$ when $\mathcal P$ is clear from context.

Given a random sample $\vb x_1,\ldots,\vb x_t$, we compute the sample moments $\hat{\vb m} = \left(\frac{1}{t}\sum_{i=1}^t\vb x_i^{\vb\alpha}\right)_{|\vb\alpha|\leq d}$, and if there exists $\hat{\vb\theta}\in\vb m_d^{-1}(\hat{\vb m})\cap\Theta$, then it serves as a point estimate of $\vb\theta$ by the law of large numbers.

\begin{definition}\label{def: mixture model}
Let $\mathcal P$ be a family of probability distributions, and suppose $\Delta_{r-1}$ is the probability simplex $\{(\lambda_1,\ldots,\lambda_r)\in[0,1]^r:\lambda_1+\cdots+\lambda_r = 1\}$.  An $r$-\textit{mixture} of $\mathcal P$ is a probability distribution of the form $\lambda_1P_{\vb\theta_1} + \cdots + \lambda_rP_{\vb\theta_r}$ for some $P_{\vb\theta_1},\ldots,P_{\vb\theta_r}\in\mathcal P$ and $\vb\lambda\in\Delta_{r-1}$.  The family of all $r$-mixtures of $\mathcal P$ is denoted $\mathcal P_r$.
\end{definition}

By linearity, a mixture has $\vb\alpha$th moment
\begin{equation*}
  E_{\vb\lambda,\vb\theta_1,\ldots,\vb\theta_r}(\vb X^{\vb\alpha}) = \lambda_1E_{\vb\theta_1}(\vb X^{\vb\alpha}) + \cdots + \lambda_rE_{\vb\theta_r}(\vb X^{\vb\alpha}),
\end{equation*}
and so $(E_{\vb\lambda,\vb\theta_1,\ldots,\vb\theta_r}(\vb X^{\vb\alpha}))_{|\vb\alpha|\leq d}$ lies in the span of $r$ points on $\mathcal M_{\mathcal P,d}$ and thus $\mathcal M_{\mathcal P_r,d}\subset\sigma_r(\mathcal M_{\mathcal P,d})$.  In fact, we have equality, cf. \cite[Proposition 4.1.1]{DSS09}.

\begin{example}\label{pearson's crab problem}
  In the landmark paper \cite{P94}, Pearson considered the mixture of two univariate Gaussian distributions and used the method of moments with $d=5$ to estimate the parameters of the distributions of the ratios of forehead breadth to body length of crabs gathered by Weldon in Naples.  In fact, if $\mathcal P$ is the family of univariate Gaussian distributions, then $\sigma_2(\mathcal M_{\mathcal P,5})$ fills its ambient space $\PP^5$, and a generic point lies on 9 secant lines to the moment variety.  Each secant line corresponds to a decomposition of the point into potential method of moments estimators, provided they satisfy the necessary conditions ($\hat{\vb\lambda}\in\Delta_1$ and $\hat{\vb\theta}_1,\hat{\vb\theta}_2\in\Theta$).  See \cite{AFS16} for more details.
\end{example}

We now recall an important statistical construction from \cite{M02}, generalized to a multiparameter family.  Here, $\partial_{\vb\theta}^{\vb\beta}$ denotes the partial derivative operator $\partial^{\vb\beta}$ with respect to affine coordinates $\vb\theta=(\theta_1,\ldots,\theta_{k'})$ on $\aff\Theta$.

\begin{definition}\label{def: local mixture}
  Let $s$ be a nonnegative integer and consider \textit{local mixing parameters} $\gamma_{\vb\beta}\in\RR$ for each vector $\vb\beta\in(\ZZ_{\geq 0})^{k'}$ with $|\vb\beta|\leq s$, with $\gamma_{\vb 0}=1$ and at least one $\vb\beta$ with $|\vb\beta|=s$ and $\gamma_{\vb\beta}$ nonzero.   Let $\mathcal P$ be a family of probability distributions such that each $P_{\vb\theta}\in\mathcal P$ has density function $\vb x\mapsto f(\vb x;\vb\theta)$ for which the likelihood $\vb\theta\mapsto f(\vb x;\vb\theta)$ is $C^s$ on $\relint\Theta$ for each $\vb x$.  If, for some $\vb\theta_0\in\relint\Theta$,
\begin{equation*}
  \vb x \mapsto \sum_{|\vb\beta|\leq s}\gamma_{\vb\beta}\partial_{\vb\theta}^{\vb\beta}f(\vb x;\vb\theta)\Big|_{\vb\theta=\vb\theta_0}
\end{equation*}
is a probability density function, then the corresponding probability distribution is a \textit{local mixture} of \textit{order} $s$ (or $s$-\textit{local mixture}) of $\mathcal P$ at $\vb\theta_0$.
\end{definition}

We denote the family of all $s$-local mixtures of $\mathcal P$ by $\mathcal P^{(s)}$ and its parameter space by $\Theta^{(s)}$, which is the set of all $(\vb\gamma,\vb\theta_0)\in\RR^{\binom{k'+s}{s}}\times\Theta$ that define an $s$-local mixture of $\mathcal P$ at $\vb\theta_0$.

\begin{example}
  Consider the family of univariate Gaussian distributions with unknown mean and fixed variance 1.  Then an $s$-local mixture at $\mu$ has density function
  \begin{equation*}
    x\mapsto \frac{1}{\sqrt{2\pi}}\sum_{\beta=0}^s\gamma_\beta\He_\beta(x - \mu)e^{-(x-\mu)^2/2},
  \end{equation*}
  where $\He_\beta$ is the $\beta$th probabilist's Hermite polynomial.  So if $s$ is odd, and thus $\He_s$ of odd degree, then this function cannot be nonnegative on all of $\RR$.  Consequently, this family only has even local mixtures.  For example, \cite[Example 4]{M02} and \cite[Example 6.6]{GKW20} use order 2 and \cite[Example 2]{AIM07} and \cite[Example 3]{MM17} use order 4.
\end{example}

\begin{definition}\label{def: location family}
  Consider a family $\mathcal P=\{P_{\vb\theta}:\vb\theta\in\RR^n\}$ of probability distributions such that $P_{\vb 0}$ has density $f$.  If the density of $P_{\vb\theta}$ is $\vb x\mapsto f(\vb x - \vb\theta)$ for all $\vb\theta$, then $\mathcal P$ is a \textit{location family}.
\end{definition}

We close this section with a key observation, generalizing ideas from \cite{AAR21,GKW20}.

\begin{theorem}\label{moment varieties of local mixtures are osculating varieties of veroneseans}
  Suppose $d$ is a positive integer, $s$ is a nonnegative integer, and $\mathcal P$ is a location family.  If $f$ is the density of $P_{\vb 0}$ and $\vb x^{\vb\alpha}\partial^{\vb\beta}f$ is integrable for all $|\vb\alpha|\leq d$ and $|\vb\beta|\leq s$, and if $\aff\Theta^{(s)}=\{\gamma_{\vb 0}=1\}\times\RR^n$, then $\mathcal M_{\mathcal P^{(s)},d}$ is projectively equivalent to $\tau_s(\nu_d(\PP^n))$.
\end{theorem}

\begin{proof}
  We first prove the $s=0$ case.  Then for each $\vb\alpha$ with $|\vb\alpha|\leq d$,
  \begin{align*}
    E_{\vb\theta}(\vb X^{\vb\alpha}) &= \int_{\RR^n}\vb x^{\vb\alpha}f(\vb x - \vb\theta)\,d\vb x = \int_{\RR^n}(\vb x + \vb\theta)^{\vb\alpha}f(\vb x)\,d\vb x = E_{\vb 0}((\vb X + \vb\theta)^{\vb\alpha}) \\
                                     &= \sum_{\vb\eta\leq\vb\alpha}\binom{\vb\alpha}{\vb\eta}\vb\theta^{\vb\eta}E_{\vb 0}(\vb X^{\vb\alpha - \vb\eta}),
  \end{align*}
  where $\vb\eta\leq\vb\alpha$ when $\eta_i\leq\alpha_i$ for all $i$ and $\binom{\vb\alpha}{\vb\eta} = \binom{\alpha_1}{\eta_1}\cdots\binom{\alpha_n}{\eta_n}$.  So we may choose an appropriate ordering $(\vb\alpha_1,\ldots,\vb\alpha_{\binom{n+d}{d}})$ so that
  \begin{equation*}
    (E_{\vb\theta}(\vb X^{\vb\alpha_i}))_i = U(\vb\theta^{\vb\alpha_i})_i,
  \end{equation*}
  where
  \begin{equation*}
    U_{ij} = \begin{cases}
      \binom{\vb\alpha_i}{\vb\alpha_j}E_{\vb 0}(\vb X^{\vb\alpha_i - \vb\alpha_j}) & \text{if }\vb\alpha_j\leq\vb\alpha_i \\
      0 &\text{otherwise}
    \end{cases}
  \end{equation*}
  is a unitriangular, and thus invertible, matrix.  But the Zariski closure $X$ of $\{(\vb\theta^{\vb\alpha_i})_i:\vb\theta\in\RR^n\}$ is projectively equivalent to $\nu_d(\PP^n)$ via the map $(\vb\theta^{\vb\alpha_i})_i\mapsto (y_0+\theta_1y_1+\cdots+\theta_ny_n)^d$.

  Now we consider positive $s$.  We write $E_{\vb\gamma,\vb\theta_0}^{(s)}(\vb X^{\vb\alpha})$ for the $\vb\alpha$th moment of an $s$-local mixture with local mixing parameters $\vb\gamma$ at $\vb\theta_0$.  Then, by our integrability hypothesis,
  \begin{equation*}
    E_{\vb\gamma,\vb\theta_0}^{(s)}(\vb X^{\vb\alpha}) = \int_{\RR^n}\vb x^{\vb\alpha}\sum_{|\vb\beta|\leq s}\gamma_{\vb\beta}\partial_{\vb\theta}^{\vb\beta} f(\vb x-\vb\theta)\Big|_{\vb\theta=\vb\theta_0}\,d\vb x 
    =  \sum_{|\vb\beta|\leq s}\gamma_{\vb\beta}\partial_{\vb\theta}^{\vb\beta} E_{\vb\theta}(\vb X^{\vb\alpha})\Big|_{\vb\theta=\vb\theta_0}
  \end{equation*}
  and so $(E_{\vb\gamma,\vb\theta_0}^{(s)}(\vb X^{\vb\alpha_i}))_i = U\vb v$, where
  \begin{equation*}
    \vb v =\left(\sum_{|\vb\beta|\leq s}\gamma_{\vb\beta}\partial_{\vb\theta}^{\vb\beta} \vb\theta^{\vb\alpha_i}\Big|_{\vb\theta=\vb\theta_0}\right)_i
  \end{equation*}
  is a point of (an affine chart of) $T_p^{(s)}X$, where $p$ has coordinates $(\vb\theta_0^{\vb\alpha_i})_i$ in this chart (see \Cref{osculating variety}).  By our hypothesis on $\Theta^{(s)}$, the closure of all such moment vectors is projectively equivalent to $\tau_s(\nu_d(\PP^n))$.
\end{proof}

\section{Terracini loci of moment varieties}\label{sec: terracini loci of moment varieties}

Much has been said about what happens when $\sigma_r(\mathcal M_{\mathcal P, d})$ is subabundant and nondefective.  A family $\mathcal P$ of probability distributions is \textit{algebraically identifiable} (from its moments of degree at most $d$) if the map $\vb m_{\mathcal P,d}$ is generically finite-to-one.  Furthermore, if $\mathcal P$ is algebraically identifiable and $\Ter_r(\mathcal M_{\mathcal P,d})\subsetneq(\mathcal M_{\mathcal P,d})^{(r)}_{sm}$, then $\mathcal P_r$ is also algebraically identifiable.  This is the case for Pearson's crab problem from \Cref{pearson's crab problem} since $\sigma_r(\mathcal M_{\mathcal P,d})$ (when $\mathcal P$ is the family of univariate Gaussians) is nondefective for all $(d,r)$ by \cite[Theorem 1]{ARS18}, and for Pearson, where $(d,r)=(5,2)$, subabundant (actually, equiabundant).  Algebraic identifiability has also been studied for mixtures of inverse Gaussian, gamma \cite{HRSY26,HSY24}, and homoscedastic Gaussian \cite{AAR21} distributions.

In light of \Cref{moment varieties of local mixtures are osculating varieties of veroneseans}, we may reinterpret the Alexander-Hirschowitz and Abo-Vannieuwenhoven theorems in the language of algebraic identifiability.

\begin{theorem}[\cite{AH95,AV18}]
  Let $\mathcal P$ be a location family of $n$-variate probability distributions satisfying the hypotheses of \Cref{moment varieties of local mixtures are osculating varieties of veroneseans}.  Then $r$-mixtures of $\mathcal P$ are algebraically identifiable from their moments of degree at most $d$ unless
  \begin{enumerate}[(a)]
  \item $d=2$, $n\geq 2$, and $r\in\left\{2,\ldots,\left\lfloor\frac{1}{n+1}\binom{n+2}{2}\right\rfloor\right\}$,
  \item $(d,n,r)\in\{(3, 4, 7), (4, 2, 5), (4, 4, 14)\}$, or
  \item $r>\frac{1}{n+1}\binom{n+d}{d}$.
  \end{enumerate}
  Furthermore, $r$-mixtures of 1-local mixtures of $\mathcal P$ are algebraically identifiable from their moments of degree at most $d$ unless
  \begin{enumerate}[(a)]
  \item $d=2$, $n\geq 6$, and $r\in\left\{2,\ldots,\left\lfloor\frac{1}{2n+1}\binom{n+2}{2}\right\rfloor\right\}$,
  \item $(d,n,r)=(3,2,2)$, or
  \item $r>\frac{1}{2n+1}\binom{n+d}{d}$.
  \end{enumerate}
\end{theorem}

Algebraic identifiability tells us about generic sets of points, but what about specific sets of points?  We define the following in the spirit of \cite{R71}.

\begin{definition}
  A family $\mathcal P$ of probability distributions with parameter space $\Theta$ is \textit{locally identifiable} (from its moments of degree at most $d$) at $\vb\theta\in\Theta$ if $\vb\theta$ is an isolated point (in the Euclidean topology on $\overline\Theta$) of $\vb m_d^{-1}(\vb m_d(\vb\theta))$.
\end{definition}

Note that if $\mathcal P$ is algebraically identifiable, then it is locally identifiable at almost all $\vb\theta\in\Theta$.  But are there any $\vb\theta$ for which it is not locally identifiable?  When dealing with mixtures, this is where the Terracini locus comes into play, as we see in this consequence of \cite[Proposition 6.2]{BC21} in the language of algebraic statistics.

\begin{theorem}
  Suppose $\mathcal P$ is a family of probability distributions.  If there exist $\vb\lambda\in\relint\Delta_{r-1}$ and $\vb\theta_1,\ldots,\vb\theta_r\in\Theta$ for which
  \begin{enumerate}[(a)]
  \item $\mathcal P$ is locally identifiable at each $\vb\theta_i$,
  \item $\mathcal P_r$ is not locally identifiable at $(\vb\lambda,\vb\theta_1,\ldots,\vb\theta_r)$, and
  \item the $[\vb m_{\mathcal P,d}(\vb\theta_i)]$ are distinct smooth points of $\mathcal M_{\mathcal P,d}$,
  \end{enumerate}
  then $\{[\vb m_{\mathcal P,d}(\vb\theta_1)],\ldots,[\vb m_{\mathcal P,d}(\vb\theta_r)]\}\in\Ter_r(\mathcal M_{\mathcal P,d})$.
\end{theorem}

\section*{Acknowledgements}
This work was supported by a grant from the Simons Foundation International (SFI-MPS-SSRFA-00012695).  The author also thanks the developers of Macaulay2 \cite{M2}, which was used extensively to explore examples and test conjectures during the early stages of this project.

\section*{Declaration of generative AI and AI-assisted technologies in the manuscript preparation process}

Generative AI (primarily Claude Opus 5.5) was used extensively for brainstorming and proof-reading this paper.  However, almost all of the text was written by the author.  The author takes full responsibility and accountability for the contents of this work.

\bibliography{tangential-terracini}{}

\begin{thebibliography}{CCGO17}

\bibitem[AAR21]{AAR21}
Daniele Agostini, Carlos Am\'endola, and Kristian Ranestad.
\newblock Moment identifiability of homoscedastic {G}aussian mixtures.
\newblock {\em Found. Comput. Math.}, 21(3):695--724, 2021.

\bibitem[AFS16]{AFS16}
Carlos Am\'endola, Jean-Charles Faug\`ere, and Bernd Sturmfels.
\newblock Moment varieties of {G}aussian mixtures.
\newblock {\em J. Algebr. Stat.}, 7(1):14--28, 2016.

\bibitem[AH95]{AH95}
J.~Alexander and A.~Hirschowitz.
\newblock Polynomial interpolation in several variables.
\newblock {\em J. Algebraic Geom.}, 4(2):201--222, 1995.

\bibitem[AH97]{AH97}
J.~Alexander and A.~Hirschowitz.
\newblock Interpolation on jets.
\newblock {\em J. Algebra}, 192(1):412--417, 1997.

\bibitem[AIM07]{AIM07}
Karim Anaya-Izquierdo and Paul Marriott.
\newblock Local mixture models of exponential families.
\newblock {\em Bernoulli}, 13(3):623--640, 2007.

\bibitem[ARS18]{ARS18}
Carlos Am\'endola, Kristian Ranestad, and Bernd Sturmfels.
\newblock Algebraic identifiability of {G}aussian mixtures.
\newblock {\em Int. Math. Res. Not. IMRN}, (21):6556--6580, 2018.

\bibitem[AV18]{AV18}
Hirotachi Abo and Nick Vannieuwenhoven.
\newblock Most secant varieties of tangential varieties to {V}eronese varieties
  are nondefective.
\newblock {\em Trans. Amer. Math. Soc.}, 370(1):393--420, 2018.

\bibitem[Bal05]{B05}
E.~Ballico.
\newblock On the secant varieties to the tangent developable of a {V}eronese
  variety.
\newblock {\em J. Algebra}, 288(2):279--286, 2005.

\bibitem[Bal22a]{B22}
Edoardo Ballico.
\newblock On the secant varieties of tangential varieties.
\newblock {\em J. Pure Appl. Algebra}, 226(12):Paper No. 107132, 24, 2022.

\bibitem[Bal22b]{B22b}
Edoardo Ballico.
\newblock Terracini loci and homogeneous spaces.
\newblock {\em Adv. Pure Appl. Math.}, 13(1):62--74, 2022.

\bibitem[Bal23]{B23}
Edoardo Ballico.
\newblock Terracini loci of codimension 1 and a criterion for partially
  symmetric tensors.
\newblock {\em Commun. Korean Math. Soc.}, 38(1):47--53, 2023.

\bibitem[Bal26]{B26}
Edoardo Ballico.
\newblock On the {H}ilbert functions of the elements of the {T}erracini loci of
  the {V}eronese varieties.
\newblock {\em Collect. Math.}, 77(1):149--160, 2026.

\bibitem[BB24a]{BB24}
Edoardo Ballico and Maria~Chiara Brambilla.
\newblock Minimal {T}erracini loci in projective space.
\newblock {\em Atti Accad. Naz. Lincei Rend. Lincei Mat. Appl.},
  35(2):175--213, 2024.

\bibitem[BB24b]{BB24b}
Edoardo Ballico and Maria~Chiara Brambilla.
\newblock Minimal {T}erracini loci in the plane: gaps and non-gaps.
\newblock {\em Rend. Circ. Mat. Palermo (2)}, 73(6):2195--2204, 2024.

\bibitem[BBF25]{BBF25}
E.~Ballico, M.~C. Brambilla, and C.~Fontanari.
\newblock Terracini loci and a codimension one {A}lexander--{H}irschowitz
  theorem.
\newblock {\em Collectanea Mathematica}, Dec 2025.

\bibitem[BBS23]{BBS23}
Edoardo Ballico, Alessandra Bernardi, and Pierpaola Santarsiero.
\newblock Terracini locus for three points on a {S}egre variety.
\newblock {\em Asian J. Math.}, 27(3):375--404, 2023.

\bibitem[BC21]{BC21}
Edoardo Ballico and Luca Chiantini.
\newblock On the {T}erracini locus of projective varieties.
\newblock {\em Milan J. Math.}, 89(1):1--17, 2021.

\bibitem[BC24]{BC24}
Edoardo Ballico and Luca Chiantini.
\newblock Terracini loci of curves.
\newblock {\em Rev. Mat. Complut.}, 37(2):447--465, 2024.

\bibitem[BCGI09]{BCGI09}
A.~Bernardi, M.~V. Catalisano, A.~Gimigliano, and M.~Id\`a.
\newblock Secant varieties to osculating varieties of {V}eronese embeddings of
  {$\Bbb P^n$}.
\newblock {\em J. Algebra}, 321(3):982--1004, 2009.

\bibitem[BV23]{BV23}
Edoardo Ballico and Emanuele Ventura.
\newblock A note on very ample {T}erracini loci.
\newblock {\em Rend. Circ. Mat. Palermo (2)}, 72(8):4115--4118, 2023.

\bibitem[CCGO17]{CCGO17}
Maria~Virginia Catalisano, Luca Chiantini, Anthony~V. Geramita, and Alessandro
  Oneto.
\newblock Waring-like decompositions of polynomials, 1.
\newblock {\em Linear Algebra Appl.}, 533:311--325, 2017.

\bibitem[CG24]{CG24}
Luca Chiantini and Fulvio Gesmundo.
\newblock Decompositions and {T}erracini loci of cubic forms of low rank.
\newblock In {\em Deformation of {A}rtinian algebras and {J}ordan type}, volume
  805 of {\em Contemp. Math.}, pages 139--155. Amer. Math. Soc., [Providence],
  RI, [2024] \copyright 2024.

\bibitem[CGG02]{CGG02}
M.~V. Catalisano, A.~V. Geramita, and A.~Gimigliano.
\newblock On the secant varieties to the tangential varieties of a
  {V}eronesean.
\newblock {\em Proc. Amer. Math. Soc.}, 130(4):975--985, 2002.

\bibitem[Chi03]{C03}
Jaydeep~V. Chipalkatti.
\newblock On equations defining coincident root loci.
\newblock {\em J. Algebra}, 267(1):246--271, 2003.

\bibitem[CLO07]{CLO07}
David Cox, John Little, and Donal O'Shea.
\newblock {\em Ideals, varieties, and algorithms}.
\newblock Undergraduate Texts in Mathematics. Springer, New York, third
  edition, 2007.
\newblock An introduction to computational algebraic geometry and commutative
  algebra.

\bibitem[DSS09]{DSS09}
Mathias Drton, Bernd Sturmfels, and Seth Sullivant.
\newblock {\em Lectures on algebraic statistics}, volume~39 of {\em Oberwolfach
  Seminars}.
\newblock Birkh\"auser Verlag, Basel, 2009.

\bibitem[EGH96]{EGH96}
David Eisenbud, Mark Green, and Joe Harris.
\newblock Cayley-{B}acharach theorems and conjectures.
\newblock {\em Bull. Amer. Math. Soc. (N.S.)}, 33(3):295--324, 1996.

\bibitem[Ger96]{G96}
Anthony~V. Geramita.
\newblock Inverse systems of fat points: {W}aring's problem, secant varieties
  of {V}eronese varieties and parameter spaces for {G}orenstein ideals.
\newblock In {\em The {C}urves {S}eminar at {Q}ueen's, {V}ol.\ {X} ({K}ingston,
  {ON}, 1995)}, volume 102 of {\em Queen's Papers in Pure and Appl. Math.},
  pages 2--114. Queen's Univ., Kingston, ON, 1996.

\bibitem[GKW20]{GKW20}
Alexandros Grosdos~Koutsoumpelias and Markus Wageringel.
\newblock Moment ideals of local {D}irac mixtures.
\newblock {\em SIAM J. Appl. Algebra Geom.}, 4(1):1--27, 2020.

\bibitem[GS]{M2}
Daniel~R. Grayson and Michael~E. Stillman.
\newblock Macaulay2, a software system for research in algebraic geometry.
\newblock Available at \url{https://macaulay2.com/}.

\bibitem[GSTT25]{GSTT25}
Francesco Galuppi, Pierpaola Santarsiero, Douglas~A. Torrance, and
  Ettore~Teixeira Turatti.
\newblock Geometry of first nonempty {T}erracini loci.
\newblock {\em Communications in Contemporary Mathematics}, 0(ja):null, 2025.

\bibitem[Har77]{H77}
Robin Hartshorne.
\newblock {\em Algebraic geometry}, volume No. 52 of {\em Graduate Texts in
  Mathematics}.
\newblock Springer-Verlag, New York-Heidelberg, 1977.

\bibitem[Hir81]{H81}
A.~Hirschowitz.
\newblock Sur la postulation g\'en\'erique des courbes rationnelles.
\newblock {\em Acta Math.}, 146(3-4):209--230, 1981.

\bibitem[Hir85]{H85}
Andr\'e Hirschowitz.
\newblock La m\'ethode d'{H}orace pour l'interpolation \`a{} plusieurs
  variables.
\newblock {\em Manuscripta Math.}, 50:337--388, 1985.

\bibitem[HRSY26]{HRSY26}
Oskar Henriksson, Kristian Ranestad, Lisa Seccia, and Teresa Yu.
\newblock Moment varieties of the inverse {G}aussian and gamma distributions
  are nondefective.
\newblock {\em J. Symbolic Comput.}, 132:Paper No. 102460, 26, 2026.

\bibitem[HSY24]{HSY24}
Oskar Henriksson, Lisa Seccia, and Teresa Yu.
\newblock Moment varieties from inverse {G}aussian and gamma distributions.
\newblock {\em Algebr. Stat.}, 15(2):329--355, 2024.

\bibitem[Kur12]{K12}
Simon Kurmann.
\newblock Some remarks on equations defining coincident root loci.
\newblock {\em J. Algebra}, 352:223--231, 2012.

\bibitem[LM24]{LM24}
Antonio Laface and Alex Massarenti.
\newblock Ample bodies and {T}erracini loci of projective varieties.
\newblock {\em Bull. Lond. Math. Soc.}, 56(1):396--410, 2024.

\bibitem[Mar02]{M02}
Paul Marriott.
\newblock On the local geometry of mixture models.
\newblock {\em Biometrika}, 89(1):77--93, 2002.

\bibitem[MM17]{MM17}
Vahed Maroufy and Paul Marriott.
\newblock Mixture models: building a parameter space.
\newblock {\em Stat. Comput.}, 27(3):591--597, 2017.

\bibitem[Pea94]{P94}
Karl Pearson.
\newblock Contributions to the mathematical theory of evolution.
\newblock {\em Philosophical Transactions of the Royal Society of London. A},
  185:71--110, 1894.

\bibitem[Rot71]{R71}
Thomas~J. Rothenberg.
\newblock Identification in parametric models.
\newblock {\em Econometrica}, 39:577--591, 1971.

\bibitem[Ter11]{T11}
Alessandro Terracini.
\newblock Sulle ${V}_k$ per cui la variet\`a degli ${S}_h(h+ 1)$-seganti ha
  dimensione minore dell'ordinario.
\newblock {\em Rendiconti del Circolo Matematico di Palermo (1884-1940)},
  31:392--396, 1911.

\end{thebibliography}
\bibliographystyle{alpha}
\end{document}